\documentclass[11pt]{amsart}
\usepackage{graphicx}
\usepackage{amssymb, amsmath}
\newtheorem{theorem}{Theorem}[section]
\newtheorem{corollary}[theorem]{Corollary}
\newtheorem{lemma}[theorem]{Lemma}
\newtheorem{proposition}[theorem]{Proposition}
\theoremstyle{definition}
\newtheorem{definition}[theorem]{Definition}

\theoremstyle{remark}

\numberwithin{equation}{section}

\begin{document}
\title{A combinatorial approach to  rational exponential groups}
\author{\sc M. Shahryari}
\thanks{{\scriptsize
\hskip -0.4 true cm MSC(2010): 20E05 and 20E10
}}

\address{M. Shahryari\\
 Department of Pure Mathematics,  Faculty of Mathematical
Sciences, University of Tabriz, Tabriz, Iran }

\email{mshahryari@tabrizu.ac.ir}
\date{\today}

\begin{abstract}
We give a suitable definition of the concept of rational complex and prove that every rational exponential group is the fundamental group of  such a complex. In this framework, we prove that the variety of rational exponential groups is a Schreier variety.
\end{abstract}

\maketitle

\section{Introduction}
A group $G$ is said to be a {\em rational exponential group}, if every element of $G$ has a unique $n$-th root for every integer $n$.
In this case, we can define a right action of the field $\mathbb{Q}$ on $G$ by $g^{\frac{m}{n}}=(g^{\frac{1}{n}})^m$. More generally,
for a given ring $A$ with identity, a group $G$ is called an {\em $A$-group}, or an {\em exponential group over $A$}, if there is  a
right action of $A$ on $G$ (denoted by $(g, \alpha)\mapsto g^{\alpha}$), such that the following axioms are satisfied:\\

1- $g^1=g$, $g^0=1$, and $1^{\alpha}=1$.

2- $g^{\alpha+\beta}=g^{\alpha}g^{\beta}$.

3- $g^{\alpha \beta}=(g^{\alpha})^{\beta}$.

4- $(hgh^{-1})^{\alpha}=hg^{\alpha}h^{-1}$.

5- if $[g, h]=1$ then $(gh)^{\alpha}=g^{\alpha}h^{\alpha}$.\\

The theory of exponential groups begins with the works of A. Mal'cev \cite{Mal}, P. Hall \cite{Hall}, G. Baumslag \cite{Baum},
R. Lyndon \cite{Lyndon}, A. Myasnikov and V. Remeslennikov \cite{MR1} and \cite{MR2}. The axioms 1-4 are introduced by Lyndon
and in this case we call $G$ an $A$-group in the sense of Lyndon. The axiom 5 is added to the definition by Myasnikov and Remeslennikov.
The class of $A$-groups in the sense of Lyndon is a variety in the algebraic language
$$
\mathcal{L}_A=(\cdot, ^{-1}, 1, f_{\alpha}:\ \alpha\in A).
$$
Here, $f_{\alpha}$ is a unitary functional symbol corresponding to
$\alpha\in A$. The class of $A$-groups (in the sense of Myasnikov
and Remeslennikov) is a quasi-variety in the language
$\mathcal{L}_A$. For any set $X$, the free exponential group on $X$
in the variety of Lyndon exponential groups will be denoted by
$F_A(X)$. It is also a member of the quasi-variety of  $A$-groups in
the restricted sense of Myasnikov-Remeslennikov. Note that, any
rational exponential group satisfies axioms 1-5, and any
$\mathbb{Q}$-group in the sense of Lyndon is a rational exponential
group, so the class of rational exponential groups is a variety.
During this article, by a normal $\mathbb{Q}$-subgroup of a rational
exponential group $G$, we mean a normal subgroup $K$ which is also a $\mathbb{Q}$-subgroup and
$$
\forall a, b\in G\ \forall m: a^mK=b^mK \Rightarrow aK=bK.
$$
In this case $G/H$ is also a rational exponential group. Note that
any rational exponential group $G$ can be expressed as
$G=F_{\mathbb{Q}}(X)/K$, where $X$ is a set and $K$ is a normal
$\mathbb{Q}$-subgroup of $F_{\mathbb{Q}}(X)$.

The theory of exponential groups is motivated by the following circumstances. First of all, there are many groups which are
 exponential: groups with unique roots of elements are $\mathbb{Q}$-groups (rational exponential groups), any group of exponent
 $m$ is a $\mathbb{Z}_m$-group. Unipotent groups over a field $K$ of characteristic zero are $K$-groups. Pro-$p$-groups are exponential
 over the ring of $p$-adic integers. Secondly, the notion of an $A$-group, generalizes the definition of $A$-modules, and so there are many
 problems of module theory, which can be investigated for exponential groups. Third, using this notion, we can define and study the
 {\em centroid} of a group $G$, which is by definition, the largest ring over which $G$ is exponential, \cite{Lio-Myas}. Finally,
 some problems of model theory concerning the universal theory of  ordinary free group, are related to the theory of exponential group.
 For example, it was conjectured (see \cite{Rem}) that any group, universally equivalent to a free group, is a subgroup of a
 free $\mathbb{Z}[x]$-group and this conjecture is proved by Kharlampovich and Myasnikov,  \cite{Olga3}, \cite{Olga4}.
 There are also many connections between exponential groups, HNN-extensions, amalgamated free products, CSA groups and torsion free
 hyperbolic groups (\cite{Olga1},  \cite{Olga2} and \cite{MR2}).

From now on, we focus on rational exponential groups. For the basic notions, the reader may use \cite{MR1} and \cite{MR2}.
Here, we give the definition of the {\em tensor $\mathbb{Q}$-completion} of an arbitrary group $E$. Let $E$ be a group and $G$ an exponential group.
Suppose there is a homomorphism $\lambda: E\to G$, with the following properties:\\

1- $G$ is $\mathbb{Q}$-generated by the image of $\lambda$,

2- for any rational group $H$ and any homomorphism $f:E\to H$, there exists a unique $\mathbb{Q}$-homomorphism $g:G\to H$,
such that $g\circ \lambda=f$.\\

Then we call the group $G$, the tensor completion of $E$. It is proved that the tensor completion exists and it is unique up
to $\mathbb{Q}$-isomorphism. So, we can use the notation $G=E^{\mathbb{Q}}$. It can be shown that $F_{\mathbb{Q}}(X)=F(X)^{\mathbb{Q}}$,
(where $F(X)$ is the ordinary free group on $X$) and  $F(X)$ is a subgroup of $F_{\mathbb{Q}}(X)$.

Note that the structure of $F_{\mathbb{Q}}(X)$ depends only on the cardinality of $X$. Although this is true for arbitrary rings, we give
here an elementary proof for the field $\mathbb{Q}$. Clearly, this argument can be used for any field of characteristic zero.

\begin{proposition}
We have $F_\mathbb{Q}(X)\cong F_\mathbb{Q}(Y)$ if and only if $|X|=|Y|$.
\end{proposition}

\begin{proof}
The "if" part is obvious. To prove the "only if" part, note that the  additive group of rational numbers is a rational exponential group,
so we may define
$$
V_X=Hom_{\mathbb{Q}}(F_{\mathbb{Q}}(X), \mathbb{Q}).
$$
Note that for $f\in V_X$, $u\in F_{\mathbb{Q}}(X)$ and $q\in \mathbb{Q}$, we have $f(u^q)=qf(u)$. We define a $\mathbb{Q}$-space structure
of $V_X$ by $(f+g)(u)=f(u)+g(u)$ and $(qf)(u)=qf(u)$. One can verify that $f+g$ and $qf$ are again elements of $V_X$ and further $V_X$ is a
vector space over $\mathbb{Q}$ with the dimension $|X|$. Now, if $\varphi: F_{\mathbb{Q}}(X)\to F_{\mathbb{Q}}(Y)$ is a $\mathbb{Q}$-isomorphism,
then $\varphi^{\ast}:V_Y\to V_X$ defined by $\varphi^{\ast}(f)=f\circ \varphi$, is an isomorphism of $\mathbb{Q}$-spaces $V_X$ and $V_Y$. So,
we  proved the assertion.
\end{proof}

A variety $\mathfrak{X}$ of algebras in an algebraic language $\mathcal{L}$ is called a {\em Schreier Variety}, if any subalgebra of any free
algebra in $\mathfrak{X}$ is again free. Many varieties are proved to be Schreier: \\

1- the variety of groups (Nielsen-Schreier, 1924, 1927),

2- the variety of non-associative algebras (Kurosh, 1947),

3- the variety of Lie algebras (Shirshov-Witt, 1953),

4- the variety of Lie $p$-algebras (Shirshov, 1953),

5- the varieties of commutative and anti-commutative algebras (Shirshov, 1954),

6- the variety of Lie superalgebras (Mikhalev-Shtern, 1985, 1986).

7- the variety of rational exponential groups (Polin, 1972).\\

For a complete discussion and citations, see \cite{Lewin}. In this paper we we give a new proof for the fact that  the variety of rational exponential groups is  a Schreier variety. We do this, using a combinatorial machinery which completely different from the approach of Polin \cite{Polin}, who uses free products in varieties of exponential groups:  we give a suitable definition of the concept of {\em rational complex}
and prove that every rational exponential group is the fundamental group of  such a complex. We use our results to  prove that the variety of
rational exponential groups is a Schreier variety. Our method is strong enough to investigate other properties of rational exponential groups, such as the subgroup theorem of Korush and ranks of $\mathbb{Q}$-subgroups of the free rational exponential group. We work over the field of rational numbers, but it can be seen that our method can be  also applied for the rings containing a copy of $\mathbb{Q}$. 

This paper consists of  the following sections: in the next section, we give  basic definitions and properties of rational complexes.
We show that the fundamental group of  a rational complex is free rational exponential group. In section 3, we introduce the notion of a
{\em colored rational complex} and in the section 4, we prove that every rational exponential group is the fundamental group of such a colored
complex. In Section 5, the notion of  the {\em covering complex} is introduced. Using this notion,  we will prove our main result: the
variety of rational exponential groups is a Schreier variety.


\section{Rational complexes}

Throughout this article, $\mathbb{Q}^{+}$ is the multiplicative group of positive rational numbers.
Any 1-dimensional complex in this article is supposed to be connected. Suppose $C$ is a 1-dimensional complex.
For any vertex $v$, the set of all cycles with the terminal point $v$, will be denoted by $C(v)$.
Suppose $\mathbb{Q}^{+}$ acts on $C(v)$. For any $p\in C(v)$ and $\alpha\in \mathbb{Q}^{+}$, we use the notation $p^{(\alpha)}$
for the result of action of $\alpha$ on $p$.

\begin{definition}
A cancelation  operation on  $C(v)$ consists of \\

1- deleting every part of the form $ee^{-1}$ or $e^{-1}e$, where $e$ is an edge.

2- replacing a part of the form $p^{(\alpha)}p^{(\beta)}$ by $p^{(\alpha+\beta)}$.

\end{definition}

By the cancelation  of type 2, one can delete any part of the form
$$
p^{(\alpha)}p^{(\beta)}(p^{(\alpha+\beta)})^{-1}
$$
in the elements of $C(v)$. We say that $p_1, p_2\in C(v)$ are homotopic, if $p_1$ can be transformed to $p_2$ by means of finitely many
cancelations. In this case we use the notation $p_1\simeq p_2$. The homotopy class of $p$ is denoted by $[p]$. We say that a cycle $p\in C(v)$ is
{\em reduced} if
$$
l(p)=\min_{q\in [p]} l(q),
$$
where $l(p)$ denotes the length of $p$ (the number of edges appearing in $p$).  It is easy to see that for any natural number $m$, we
have $p^{(m)}\simeq p^m$. Note that, we don't know how many reduced elements are there in a given homotopy class, however, we will
assume in the sequel that this element is unique for a special kind of classes (classes of atomic cycles). A cycle $p$ will be called
{\em cyclically reduced}, if it is reduced and it is not  in the form $uqu^{-1}$, where $u$ and $q$ are cycles.

\begin{definition}
Let $C$ be a 1-dimensional complex with  an action of $\mathbb{Q}^{+}$ on every $C(v)$. Suppose further that\\

1- for any vertex $v$, any $p_1, p_2\in C(v)$ and any natural number $m$
$$
p_1^m\simeq p_2^m \Rightarrow p_1\simeq p_2,
$$

2- for any $p\in C(v)$ and $\alpha\in \mathbb{Q}^{+}$
$$
(p^{-1})^{(\alpha)}\simeq (p^{(\alpha)})^{-1},
$$

3- for any two vertices $v$ and $v^{\prime}$, any path $u$ from $v$ to $v^{\prime}$, any $p\in C(v^{\prime})$ and $\alpha \in \mathbb{Q}^{+}$
$$
(upu^{-1})^{(\alpha)}\simeq up^{(\alpha)}u^{-1}.
$$

Then we say that  $C$ has a {\em compatible action} of $\mathbb{Q}^{+}$.
\end{definition}

\begin{definition}
Let $C$ be a 1-dimensional complex with a compatible action of $\mathbb{Q}^{+}$ and suppose for any vertex $v\in C$, there is a {\em height function}
$h_v: C(v)\to \mathbb{Z}$ with the following properties:\\

1- Its image is an infinite  set of  non-negative  integers, $h_v(p)=0$ if and only if $p\simeq 1$, where $1$ is the trivial cycle,
and $p\simeq q$ implies $h_v(p)=h_v(q)$.

2- If $m>1$, and if there is no $q$ such that $p\simeq q^m$ and $h_v(q)=h_v(p)$, then
$$
h_v(p^{(\frac{1}{m})})=h_v(p)+1.
$$

3- We have the inequality
$$
h_v(p_1p_2\ldots p_k)\leq \max_{i} h_v(p_i),
$$
and equality holds, if $p_1p_2\ldots p_k$ is  reduced.

4- If $h_v(p)$ is not equal to the minimum of $h_v$, then there are cycles $q_1, \ldots, q_k$ and integers $m_1, \ldots, m_k$, such that
$$
p\simeq q_1^{(\frac{1}{m_1})}q_2^{(\frac{1}{m_2})}\ldots q_k^{(\frac{1}{m_k})},
$$
and $h_v(q_i)< h_v(p)$, for any $i$. Further, if $k$ is minimum, then the expression is unique up to homotopy; this means that if we have also
$$
p\simeq u_1^{(\frac{1}{r_1})}u_2^{(\frac{1}{r_2})}\ldots u_k^{(\frac{1}{r_k})}
$$
with $h_v(u_i)<h_v(p)$, then for any $i$, $q_i\simeq u_i$ and $m_i=r_i$.

Before going to the next items, we introduce the following notion: a cycle $p\in  C(v)$ is {\em atomic} if its height is equal
the minimum positive value of $h_v$.

5- If $p$ is atomic and cyclically reduced, then $p^m$ is also reduced for any integer $m$.

6- If $p$ and $q$ are atomic and reduced, then $p\simeq q$ implies $p=q$.

Then we call $C$ a {\em rational complex}.
\end{definition}

Let $C$ be a rational complex and $v$ be any vertex of $C$. Let
$$
\pi_{\mathbb{Q}}(C, v)=\{ [p]:\ p\in C(v)\}.
$$
Define a binary operation on this set by $[p][q]=[pq]$. It is easy to see that $\pi_{\mathbb{Q}}(C, v)$ is a rational exponential group,
which we call it the {\em fundamental group} of $C$ based at $v$. Note that since $C$ is connected, so for any two different vertices $v$
and $v^{\prime}$, we have
$$
\pi_{\mathbb{Q}}(C, v)\cong  \pi_{\mathbb{Q}}(C, v^{\prime}),
$$
and further this is an isomorphism of rational groups. The main aim of this section is to prove that this fundamental group is free in the
variety of rational exponential groups. In what follows, a  rational complex with a unique vertex is called a {\em rational bouquet}.

\begin{definition}
Let $C$ be a rational complex and $T\subseteq C$ be a maximal tree. We define a rational bouquet $C_T$ with the following properties:\\

1- It has a unique vertex $v$ (which is at the same time a fixed vertex of $C$).

2- For any edge $e\in C\setminus T$, there is a unique edge $\hat{e}\in C_T$.

3- If $\widehat{e_1}=\widehat{e_2}$, then $e_1=e_2$.

4- We have $\widehat{e^{-1}}=\hat{e}^{-1}$.

To define the action of $\mathbb{Q}^+$ on $C_T(v)$, we need a notation. First we put $\hat{e}=v$ for $e\in T$. Second, if $p=e_1e_2\ldots e_n$,
then we define $p_T=\widehat{e_1}\widehat{e_2}\ldots\widehat{e_n}$. Now we can continue the definition.\\

5- The action of $\mathbb{Q}^+$ is given by $(p_T)^{(\alpha)}=(p^{(\alpha)})_T$.

6- The height function is $h(p_T)=h_v(p)$.

We call $C_T$ the {\em retract} of $C$ with respect to $T$.

\end{definition}

Clearly, we should check the requirements of a rational complex for $C_T$. This can be done, if we prove the following lemma.
We also need the lemma, if we want to show that the action of $\mathbb{Q}^+$ and the height function of $C_T$ are well-defined.

\begin{lemma}
Let $p$ and $q$ be elements of $C(v)$. We have\\

1- $p_T=q_T$ if and only if $p$ and $q$ are the same except in some possible  parts of the form $e_i\ldots e_je_j^{-1}\ldots e_i^{-1}$,
with $e_i, \ldots e_j\in T$.

2- $p_T\simeq q_T$ if and only if $p\simeq q$.
\end{lemma}

\begin{proof}
Suppose $p_T=q_T$ and
$$
p=e_1e_2\ldots e_n\ \ e_{i_1}, \ldots, e_{i_r}\in C\setminus T
$$
$$
q=f_1f_2\ldots f_m\ \ f_{j_1}, \ldots, f_{j_s}\in C\setminus T,
$$
so, we have $\widehat{e_{i_1}}\ldots \widehat{e_{i_r}}=\widehat{f_{j_1}}\ldots\widehat{f_{j_s}}$. Hence  $r=s$ and for any $t$, $e_{i_t}=f_{j_t}$.
There are three possibilities:

Case i: Let $i_1, j_1>1$. Then $e_1\ldots e_{i_1-1}\subseteq T$ and $f_1\ldots f_{j_1-1}\subseteq T$. So, $e_1\ldots e_{i_1-1}f_{j_1-1}^{-1}\ldots f_1^{-1}$ is a cycle in $T$, and this is impossible, since $T$ is a tree. Hence we must have $e_1=f_1$, \ldots, $e_{i_1-1}=f_{j_1-1}$.

Case ii: Let $i_1>1$ and $j_1=1$. Then $e_1\ldots e_{i_1-1}\subseteq T$ and it is a cycle.

Case iii: Let $i_1=j_1=1$. In this case we have again three possibilities for $i_2$ and $j_2$, and so the same argument begins.

Therefore $p$ and $q$ are the same except in some possible  parts of the form $e_i\ldots e_je_j^{-1}\ldots e_i^{-1}$, with $e_i, \ldots e_j\in T$.
This proves the assertion 1. To prove the second part, we show that $p\simeq 1$ iff $p_T\simeq 1$. Clearly, if $ee^{-1}$ appears in $p$ and $e\in T$,
then automatically $\hat{e}\hat{e}^{-1}=v$. If $e\in C\setminus T$, then $\hat{e}\hat{e}^{-1}$ is a part of $p_T$. So, any cancelation of the
first type in $p$, causes a similar cancelation in $p_T$. Suppose
$$
q^{(\alpha)}q^{(\beta)}=(e_i\ldots e_j)(e_{j+1}\ldots e_k)
$$
is a part of $p$. Then clearly
\begin{eqnarray*}
(q_T)^{(\alpha)}(q_T)^{(\beta)}&=&(q^{(\alpha)})_T(q^{(\beta)})_T\\
                               &=&(\widehat{e_i}\ldots \widehat{e_j})(\widehat{e_{j+1}}\ldots \widehat{e_k})
\end{eqnarray*}
is a part of $p_T$ and so any cancelation of the second type in $p$ implies a similar cancelation for $p_T$. Therefore,
we showed that $p\simeq 1$ implies $p_T\simeq 1$. We prove the converse. Let $\hat{e}\hat{e}^{-1}$ be  a part of $p_T$.
So, $e$ does not lie in $T$ and $p$ has a part of the form $ee^{-1}$ or a part of the form $ee_1\ldots e_je^{-1}$, where $e_1, \ldots, e_j\in T$.
In the second case we have
$$
end(e)=in(e^{-1})=end(e_j)\ \ and \ \ end(e)=in(e_1),
$$
so $e_1\ldots e_j$ is a cycle in $T$, which is a contradiction. Hence $ee^{-1}$ is a part of $p$. If $(q_T)^{(\alpha)}(q_T)^{(\beta)}$
appears in $p_T$, then $q^{(\alpha)}q^{(\beta)}$ will be a part of $p$ or a part of the form $q^{(\alpha)}e_1\ldots e_jq^{(\beta)}$ will
appear in $p$. In the second case, $e_1\ldots e_j$ is again a cycle in $T$. This argument shows that if $p_T\simeq 1$, then $p\simeq 1$.

\end{proof}

\begin{theorem}
$C_T$ is a rational bouquet and we have
$$
\pi_\mathbb{Q}(C, v)\cong \pi_\mathbb{Q}(C_T, v).
$$
\end{theorem}

\begin{proof}
Let $m$ be an integer and $(p_T)^m\simeq (q_T)^m$. Then we have $(p^m)_T\simeq (q^m)_T$ and by the lemma $p^m\simeq q^m$.
Since $C$ is a rational complex, we have $p\simeq q$ and again by the lemma above  $p_T\simeq q_T$. Other parts of the definition of a
rational complex and requirements of height function can be checked easily, for example we verify the conditions of height function. \\

1- The images of $h$ and $h_v$ are the same, and if $h(p_T)=0$ then $h_v(p)=0$, so $p\simeq 1$. This shows that $p_T\simeq 1$.
Also, if $p_T\simeq q_T$, then $p\simeq q$ and hence $h_v(p)=h_v(q)$, showing that $h(p_T)=h(q_T)$.

2- Let $m>1$ and suppose there is no $q_T$ such that $p_T\simeq q_T^m$ and $h(q_T)=h(p_T)$. Then, also there is no $q$, with
the property $p\simeq q^m$ and $h_v(q)=h_v(p)$. Hence
$$
h_v(p^{(\frac{1}{m})})=h_v(p)+1,
$$
and therefore
$$
h(p_T^{(\frac{1}{m})})=h(p_T)+1.
$$

3- We have
$$
h((p_1)_T(p_2)_T\ldots (p_k)_T)=h_v(p_1p_2\ldots p_k)\leq \max_i h_v(p_i)=\max_i h((p_i)_T),
$$
and the equality holds if $(p_1)_T(p_2)_T\ldots (p_k)_T$ is  reduced.

4- Suppose $h(p_T)$  is not minimum. Then clearly, $h_v(p)$ is also not minimum, and so $p_T$ can be expressed in terms of cycles of lower height.
Suppose we have
$$
p_T\simeq (q_1)_T^{(\frac{1}{m_1})}(q_2)_T^{(\frac{1}{m_2})}\ldots (q_k)_T^{(\frac{1}{m_k})}=(u_1)_T^{(\frac{1}{r_1})}(u_2)_T^{(\frac{1}{r_2})}\ldots (u_k)_T^{(\frac{1}{r_k})},
$$
where for all $i$, $h((q_i)_T), h((u_i)_T)<h(p_T)$ and $k$ is minimum. Then we deduce by the above lemma that
$$
p\simeq q_1^{(\frac{1}{m_1})}q_2^{(\frac{1}{m_2})}\ldots q_k^{(\frac{1}{m_k})}\simeq u_1^{(\frac{1}{r_1})}u_2^{(\frac{1}{r_2})}\ldots u_k^{(\frac{1}{r_k})}.
$$
This shows that $(q_i)_T\simeq (u_i)_T$ and $m_i=r_i$.

5- Let $p_T$ be cyclically reduced and atomic. Then clearly $p$ is also atomic. Suppose $p_T=\widehat{e_1} \ldots \widehat{e_n}$. Then we have
$$
p=u_1e_1w_1u_2e_2w_2\ldots u_ne_nw_n,
$$
where all $u_i$ and $w_i$ are pathes (not cycles) in $T$. If $p$ is not reduced, then there must be a cancelation of the second type.
So, let $q^{(\alpha)}$ be a part of $p$. Then $q^{(\alpha)}=u_ie_iw_i\ldots u_je_jw_j$, for some $i$ and $j$. Hence
$$
(q^{(\alpha)})_T=\widehat{e_i}\ldots \widehat{e_j}
$$
is a part of $p_T$, a contradiction. So, $p$ is reduced. Finally, $p$ is cyclically reduced, since, if it has the form $uqu^{_1}$, then at
the same time, $u$ will be a cycle and a part of $T$ (note that $u$ is a cycle).

6- Let $p_T$ and $q_T$ be atomic and reduced. Then $p$ and $q$ are atomic and reduced. So, if $p_T\simeq q_T$, then by the lemma,
$p\simeq q$ and hence $p=q$. So, we have $p_T=q_T$.

Finally, it is also easy to check that the map $[p]\mapsto [p_T]$ is an isomorphism between the fundamental groups of $C$ and $C_T$.
\end{proof}

The above result shows that in the investigation of the fundamental group of rational complexes, it is enough to concentrate just on
rational bouquets. In what follows $\min C$ will denote the minimum height of non-trivial elements of $C(v)$.

The proof of the next lemma is quite easy. Recall that an atomic edge in a rational bouquet $(C, v)$ is an edge $e$, such that $h(e)=\min C$.
By the requirements of the height function, we see that for an atomic edge $e$, the inverse edge $e^{-1}$ is also atomic.

\begin{lemma}
If $(C, v)$ is a rational bouquet, then it contains an atomic edge. If $e_1, \ldots, e_n$ are atomic edges, then the cycle $e_1e_2\ldots e_n$
is also atomic.
\end{lemma}

\begin{lemma}
Suppose
$$
F=\langle [e]: e\in C \ \ and \ \ e\ is\ atomic\ edge\rangle.
$$
Then $F$ is a free group.
\end{lemma}

\begin{proof}
If this is not the case, then there is a relation of the form
$$
e_1^{m_1}e_2^{m_2}\ldots e_n^{m_n}\simeq 1,
$$
where each $e_i$ is atomic, any $m_i$ is a positive integer and the sum $m_1+m_2+\cdots+m_n$ is minimum. We have
$$
e_1^{m_1}\simeq e_n^{-m_n}\ldots e_2^{-m_2}.
$$
The right hand side is reduced, since if it is not reduced, then
$$
e_2^{m_2}\ldots e_n^{m_n}\simeq f_1^{r_1}\ldots f_k^{r_k},
$$
where the right hand side is reduced and its edges are atomic. Also, we have $r_1+\cdots+r_k<m_2+\cdots+m_n$. Now, we have
$$
e_1^{m_1}f_1^{r_1}\ldots f_k^{r_k}\simeq 1,
$$
contradicting the minimality of $m_1+m_2+\cdots+m_n$. So, by items 5 and 6 in the definition of rational complex, we have
$$
e_1^{m_1}=e_n^{-m_n}\ldots e_2^{-m_2},
$$
and this implies for example $e_1=e_2^{-1}$ which is a contradiction. Hence, $F$ is an ordinary free group.
\end{proof}

\begin{corollary}
Let $C$ be a rational bouquet with the vertex $v$. Then $\pi_{\mathbb{Q}}(C, v)$ is free in the variety of rational exponential groups.
\end{corollary}

\begin{proof}
Suppose
$$
F=\langle [e]: e\in C \ \ and \ \ e\ is\ atomic\ edge\rangle.
$$
As we saw, $F$ is a free group. We prove that
$$
\pi_{\mathbb{Q}}(C, v)\cong F^{\mathbb{Q}}.
$$
Let $\lambda:F\to \pi_{\mathbb{Q}}(C, v)$ be the inclusion map. Clearly $F$ $\mathbb{Q}$-generates $\pi_{\mathbb{Q}}(C, v)$ as a $\mathbb{Q}$-group.
Let $H$ be any rational group and suppose $f:F\to H$ is a homomorphism. We show that $f$ can be extended to whole group $\pi_{\mathbb{Q}}(C, v)$.
Define $g:\pi_{\mathbb{Q}}(C, v)\to H$ by induction on $h_v(p)$. If $h_v(p)=\min C$ then put $g([p])=f([p])$. Let $h_v(p)>\min C$.
We know that there is a unique expression
$$
p\simeq q_1^{(\frac{1}{m_1})}q_2^{(\frac{1}{m_2})}\ldots q_k^{(\frac{1}{m_k})},
$$
with $h_v(q_i)<h_v(p)$, such that $k$ is minimum. So we can define
$$
g([p])=g([q_1])^{\frac{1}{m_1}}g([q_2])^{\frac{1}{m_2}}\ldots g([q_k])^{\frac{1}{m_k}}.
$$
Now, we have a $\mathbb{Q}$-homomorphism $g$ such that $g\circ\lambda=f$. Hence $\pi_{\mathbb{Q}}(C, v)\cong F^{\mathbb{Q}}$.
On the other hand we know that $F^{\mathbb{Q}}$ is free in the variety of rational exponential groups.
\end{proof}


\section{Colored rational complexes}

In this section, we extend our definitions, in order that any rational exponential group can be represented as the fundamental group of a (colored)
rational complex.

\begin{definition}
Let $C$ be a rational complex and $v$ be an arbitrary vertex. Let $\Phi\subseteq C(v)$. Then the triple $(C, \Phi, v)$ is called a
{\em colored rational complex}.
\end{definition}

Suppose  $v^{\prime}$ is another vertex and $u$ is a path from $v^{\prime}$ to $v$. Let
$$
\Phi^{\prime}=\{ upu^{-1}: p\in \Phi\}\subseteq C(v^{\prime}).
$$
Then $(C, \Phi^{\prime}, v^{\prime})$ is also a colored rational complex.

\begin{definition}
Let $(C, \Phi, v)$ be a colored rational complex. We introduce a cancelation of the third type as follows:
\begin{center}
any part of the form $p\in \Phi$ can be canceled.
\end{center}
Two cycles $p_1$ and $p_2$ are called {\em color homotopic}, if we can transform one of them to another by a number of cancelation operations of
the first, second or third types. In this case, we write $p_1\simeq_{\ast} p_2$. The colored homotopy class of $p$ is denoted by $[p]_{\ast}$.

\end{definition}

From now on, in the case of colored rational complexes, we use the stronger  assumption
$$
p_1^m\simeq_{\ast}p_2^m\Rightarrow p_1\simeq_{\ast} p_2,
$$
rather than
$$
p_1^m\simeq p_2^m\Rightarrow p_1\simeq p_2.
$$

\begin{definition}
We define the fundamental group
$$
\pi_{\mathbb{Q}}(C, \Phi, v)=\{ [p]_{\ast}:\ p\in C(v)\}.
$$
Clearly, this is a rational exponential group and in fact, we have
$$
\pi_{\mathbb{Q}}(C, \Phi, v)=\frac{\pi_{\mathbb{Q}}(C, v)}{K},
$$
where $K$ is the normal $\mathbb{Q}$-subgroup generated by  the set
$$
\{ [p]_{\ast}:\ p\in \Phi\}.
$$
\end{definition}

\begin{definition}
Let $(C, \Phi, v)$ be a colored rational complex and $T\subseteq C$ be a maximal tree. Let $C_T$ be the retract of $C$ with respect
to $T$ and the vertex $v$. Suppose
$$
\Phi_T=\{ p_T:\ p\in \Phi\}.
$$
Then $(C_T, \Phi_T, v)$ is called the retract of $(C, \Phi, v)$ with respect to $T$.
\end{definition}

Obviously, we must prove that this retract is really a colored rational complex. In the next lemma, we do this.

\begin{lemma}
$(C_T, \Phi_T, v)$ is a colored rational complex.
\end{lemma}

\begin{proof}
We know that $(C_T, v)$ is a rational bouquet. So, we must prove the following implication
$$
(p_T)^m\simeq_{\ast} (q_T)^m \Rightarrow p_T\simeq_{\ast}q_T.
$$
Equivalently, we prove that if $(p_T)^m\simeq_{\ast}1$, then $p_T\simeq_{\ast} 1$. Remember that the cancelation operations
of types first and second are verified before, so we show that if $p_1\in \Phi$ is a part of $p$, then $(p_1)_T$ is a part of $p_T$ and
the converse is also true in some sense. The first part is trivial, so, let $(p_1)_T$ is a part of $p_T$. Let $p=e_1e_2\dots e_n$ and
$p_T=\widehat{e_{j_1}}\ldots \widehat{e_{j_t}}$, where
$$
1\leq j_1\leq \cdots\leq j_t\leq n,
$$
and $e_{j_i}\in C\setminus T$. We prove that there is a cycle $q$, colore homotopic to $p$, such that $p_1$ appears in $q$. Let
$$
(p_1)_T=\widehat{e_{j_r}}\ldots \widehat{e_{j_s}},
$$
and suppose $p_1=f_1f_2\ldots f_l$. So, for any $r\leq t\leq s$, there exists an index $i_t$ such that $e_{j_t}=f_{i_t}$. Now, for any $t$, the path
$$
e_{j_t+1}\ldots e_{j_{t+1}-1}f_{i_{t+1}-1}^{-1}\ldots f_{i_t+1}^{-1}
$$
is a cycle in $T$, and this is not possible, except in the case
$$
e_{j_t+1}=f_{i_t+1}, \ldots, e_{j_{t+1}-1}=f_{i_{t+1}-1}.
$$
In other word
$$
f_{i_r}f_{i_r+1}\ldots f_{i_s}=e_{j_r}e_{j_r+1}\ldots e_{j_s}.
$$
If $e_1\ldots e_{j_r-1}\neq f_1\dots f_{i_r-1}$, then we put
$$
q_1= e_1\ldots e_{j_r-1}f_{i_r-1}^{-1}\ldots f_1^{-1},\ \ and\ \ q_2=e_n^{-1}\ldots e_{j_s+1}^{-1}f_{i_s+1}\ldots f_l.
$$
Then we have $p\simeq_{\ast}q_1p_1q_2$ and so $p_1$ is a part of some colore homotop of $p$. Now, we can complete the proof as follows: Let
$$
(p_T)^m\simeq_{\ast} (q_T)^m.
$$
Then we have
$$
(p^m)_T\simeq_{\ast} (q^m)_T,
$$
and therefore $p^m\simeq_{\ast}q^m$. Since $C$ is a rational complex, we must have $p\simeq_{\ast}q$. So, $p_T\simeq_{\ast}q_T$.
\end{proof}

We close this section with a generalization of Theorem 2.6.

\begin{theorem}
Let $(C, \Phi, v)$ be a colored rational complex and $T$ be a maximal tree. Then we have
$$
\pi_{\mathbb{Q}}(C, \Phi, v)\cong \pi_{\mathbb{Q}}(C_T, \Phi_T, v).
$$
\end{theorem}

Note that in order to prove this theorem, one has to repeat the proof of Theorem 2.6 with respect to "$\simeq_{\ast}$". In particular, one needs
$$
(p_T)^m\simeq_{\ast}(q_T)^m\Rightarrow p_T\simeq_{\ast}q_T,
$$
proved in Lemma 3.5.

\section{Exponential groups as fundamental groups}
In Section 2, we showed that for any rational complex $(C, v)$, the fundamental group $\pi_{\mathbb{Q}}(C, v)$ is free in the
variety of rational exponential groups. In this section, we prove that the converse is also true. More generally, we show that
any rational exponential group is the fundamental group of some colored rational complex.

In this section, $X$ is an arbitrary set and $F_{\mathbb{Q}}(X)$ is the free  rational exponential group, generated by $X$.
Let $A_0=X^{-1}\cup X=X^{\pm}$. Let
$$
A_{n+1}=\{ w^{\frac{1}{m}}:\ 2\leq m, w\in \langle A_n\rangle \},
$$
where $\langle A_n\rangle$ is the ordinary subgroup of $F_{\mathbb{Q}}(X)$, generated by $A_n$.

\begin{lemma}
We have $A_0\subseteq A_1\subseteq A_2\subseteq \cdots$, and $F_{\mathbb{Q}}(X)=\cup_nA_n$.
\end{lemma}

\begin{proof}
Let $w\in A_n$ and $m\geq 2$. Then $w^m\in \langle A_n\rangle$ and hence
$$
w=(w^m)^{\frac{1}{m}}\in A_{n+1}.
$$
So, $A_n\subseteq A_{n+1}$. Let $w\in F_{\mathbb{Q}}(X)$. We have
$$
w=(w_1\ldots w_k)^{\frac{r}{s}},
$$ where $r, s\geq 1$ and by induction, every $w_i$ belongs to some $A_{n_i}$. Therefore, if $n$ is large enough, then
$w_1, \ldots, w_k\in A_n$, and hence
$$
(w_1\ldots w_k)^{2r}\in \langle A_n\rangle.
$$
Now, we have
$$
w=((w_1\ldots w_k)^{2r})^{\frac{1}{2s}}\in A_{n+1}.
$$
\end{proof}

For $w\in A_n\setminus A_{n-1}$, we define the height of $w$ to be $h(w)=n+1$. Elements of $A_0$ have height one. Let $B_0=A_0$ and
$B_n=A_n\setminus \langle A_{n-1}\rangle$, for $n>0$. Put $B=\cup_nB_n$. Note that $B$ is a proper subset of $F_{\mathbb{Q}}(X)$.
Every element of $B_n$ is called a {\em pure radical} of height $n+1$. Obviously, for positive $n$, such an element has the
form $w=u^{\frac{1}{m}}$, where $m\geq 2$ and $u\in \langle A_{n-1}\rangle$ is not a $m$-th power (in the subgroup $\langle A_{n-1}\rangle$).
Note that, for positive $n$, the set $B_n$ is closed under the operations of root extracting. Using  pure radicals, we may define the reduced
form of the elements of $F_{\mathbb{Q}}(X)$.

\begin{definition}
Reduced forms of the elements of $F_{\mathbb{Q}}(X)$ are defined inductively as follows:\\

1- Elements of $A_0$ are reduced.

2- A pure radical $u^{\frac{1}{m}}$ is reduced, if $u$ is reduced and $h(u)<h(u^{\frac{1}{m}})$.

3- A product $w_1\ldots w_k$ is reduced, if every $w_i$ is a reduced pure radical and $k$ is minimum.
\end{definition}

Note that in number 3, we have automatically $h(w_i)\leq h(w_1\ldots w_k)$. If this is not the case, then $w_i$ will be canceled and this
contradicts the minimality of $k$. Every element of $F_{\mathbb{Q}}(X)$ can be represented in the reduced form. For example, the reduced
form of $(xy)(y^{-1}x)^{\frac{1}{2}}(xy)^{-1}$ is $(x^2y^{-1}x^{-1})^{\frac{1}{2}}$. We will prove that the reduced form of an element is unique.

\begin{lemma}
Suppose $w_1, \ldots, w_k\in B$, such that $w_iw_{i+1}\neq 1$, for any $i$ and $w_1\ldots w_k\in B$. Then, there exists an element $u\in B$ and
there are rational numbers $\alpha_1, \ldots, \alpha_k$, such that $w_i=u^{\alpha_i}$, for any $i$.
\end{lemma}

\begin{proof}
Let $w=w_1\ldots w_k$. We apply induction on the height of $w$. The assertion is clear, if $h(w)=1$. So, let $h=h(w)>1$. We have $w=u^{\frac{1}{m}}$,
where $m\geq 2$, and $h(u)=h-1$. Note that one can write $u=u_1\ldots u_r$, such that any $u_i$ is pure radical of height at most $h-1$. Now,
$$
u_1=(w_1\ldots w_k)^mu_r^{-1}\ldots u_2^{-1},
$$
and if $w_k=w_1^{-1}$, then
$$
u_1=w_1\ldots w_{k-1}w_2\ldots w_{k-1}\ldots w_ku_r^{-1}\ldots u_2^{-1}.
$$
So, after cancelation, still any $w_i$ appears in the right hand side. If $w_{k-1}=w_2^{-1}$, then again after cancelation, any $w_i$
appears in the right hand side. This argument shows that in the case $w_k\neq u_r$, the assertion follows from the induction. Therefore,
suppose $w_k=u_r$. Then
$$
u_1\ldots u_{r-1}w_k=(w_1\ldots w_k)^{m-1}w_1\ldots w_k,
$$
so
$$
u_1\ldots u_{r-1}=(w_1\ldots w_k)^{m-1}w_1\ldots w_{k-1}.
$$
Note that, it is impossible to delete every $u_i$, since there is at least one $w_i$ of height $h$, while the height of every $u_i$ is at
most $h-1$. Hence, there are $i$ and $j$ such that
$$
u_1\ldots u_i=(w_1\ldots w_k)^{m-1}w_1\ldots w_j.
$$
Again, note that there is no $s>1$, with the property $u_s\ldots u_r=w_1\ldots w_k$, since otherwise, $u_s\ldots u_r=w$, a contradiction.
Hence, we have
$$
u_1=(w_1\ldots w_k)^{m-1}w_1\ldots w_ju_i^{-1}\ldots u_2^{-1},
$$
so, we may apply induction, to get the assertion.
\end{proof}

\begin{corollary}
The reduced form of an arbitrary element in $F_{\mathbb{Q}}(X)$ is unique.
\end{corollary}

\begin{proof}
Suppose $w\in F_{\mathbb{Q}}(X)$ and
$$
w=w_1w_2\ldots w_k=w^{\prime}_1w^{\prime}_2\ldots w^{\prime}_k,
$$
are two reduced forms of $w$. We use induction on $k$. The case $k=1$ is trivial. So, let $k>1$. We have
$$
w_k=w_{k-1}^{-1}\ldots w_1^{-1}w^{\prime}_1w^{\prime}_2\ldots w^{\prime}_k.
$$
If $w_1=w^{\prime}_1$, then we may apply the induction to complete the proof. So, suppose $w_1\neq w^{\prime}_1$. Therefore, we have
$$
w_i^{-1}w_{i-1}^{-1}\neq 1, w_1^{-1}w^{\prime}_1\neq 1, w^{\prime}_i w^{\prime}_{i+1}\neq 1,
$$
so, we using the previous lemma, we conclude that there is a $u\in B$ and there are rational numbers
$$
\alpha_1, \ldots, \alpha_{k-1}, \beta_1, \ldots, \beta_k
$$
such that
$$
w_1=u^{\alpha_1}, \ldots, w_{k-1}=u^{\alpha_{k-1}}\ \ and\ \ w^{\prime}_1=u^{\beta_1}, \ldots, w^{\prime}_k=u^{\beta_k}.
$$
This contradicts the minimality of $k$.
\end{proof}

Now, we are ready to associate a colored rational complex to every
rational exponential group. Let $G$ be such a group. There is a set
$X$ and a normal $\mathbb{Q}$-subgroup $K$ in $F_{\mathbb{Q}}(X)$,
such that $G\cong F_{\mathbb{Q}}(X)/K$. We define a colored rational
complex $C(X, K)=(C, \Phi, v)$ as follows:\\

1- $C$ has a unique vertex $v$.

2- For any pure radical $w\in B$, there is a corresponding edge $e(w)$ in $C$.

3- If $p=e(w_1)e(w_2)\ldots e(w_k)$ and $\alpha\in \mathbb{Q}^+$, and if $w^{\prime}_1w^{\prime}_2\ldots w^{\prime}_r$ is the reduced form of
$(w_1w_2\ldots w_k)^{\alpha}$, then we define
$$
p^{(\alpha)}=e(w^{\prime}_1)e(w^{\prime}_2)\ldots e(w^{\prime}_r).
$$

4- If $p=e(w_1)e(w_2)\ldots e(w_k)$ and $w^{\prime}$ is the  reduced form of $w_1\ldots w_k$, then the height of $p$ is equal to $h(w^{\prime})$.

5- We have
$$
\Phi=\{ e(w_1)e(w_2)\ldots e(w_k):\ w_1w_2\ldots w_k\in K\}.
$$

In the sequel, we show that $C(X, K)$ is really a colored rational complex and further its fundamental group is isomorphic to $G$.
From now on, we use the notation $\partial(w)$ for the cycle $e(w_1)\ldots e(w_k)$, where $w_1\ldots w_k$ is the reduced form of $w$.
In this notation, we have
$$
(e(w_1)\ldots e(w_k))^{(\alpha)}=\partial((w_1\ldots w_k)^{\alpha}),
$$
and
$$
\Phi=\{ \partial(w):\ w\in K\}.
$$
The cancelation operations on $C(X,K)$ can  be restated as follows:\\

1- Deleting any part of the form $e(w)e(w^{-1})$, with $w\in B$.

2- Replacing parts of the form $\partial(w^{\alpha})\partial(w^{\beta})$, with $\partial(w^{\alpha+\beta})$.

3- Deleting any part of the form $\partial(w)$, with $w\in K$.\\

\begin{lemma}
Let $w_1, \ldots, w_k, w^{\prime}_1, \ldots, w^{\prime}_l\in B$ and
$$
w_1\ldots w_k=w^{\prime}_1\ldots w^{\prime}_l.
$$
Then we have
$$
e(w_1)\ldots e(w_k)\simeq e(w^{\prime}_1)\ldots e(w^{\prime}_l).
$$
\end{lemma}

\begin{proof}
We may assume that $w_iw_{i+1}\neq 1$, $w^{\prime}_iw^{\prime}_{i+1}\neq 1$ and also $w_1\neq w^{\prime}_1$. We have
$$
w_k=w_{k-1}^{-1}\ldots w_1^{-1}w^{\prime}_1\ldots w^{\prime}_l,
$$
so by  Lemma 4.3, there is $u\in B$ and there are rational numbers
$$
\alpha_1, \ldots, \alpha_{k-1}, \beta_1, \ldots, \beta_l,
$$
such that
$$
w_1=u^{\alpha_1}, \ldots, w_{k-1}=u^{\alpha_{k-1}}, w^{\prime}_1=u^{\beta_1}, \ldots, w^{\prime}_l=u^{\beta_l}.
$$
Since, $u$ is a pure radicals, we have
$$
\partial(u^{\alpha_i})=e(u^{\alpha_i}), \ \partial(u^{\beta_i})=e(u^{\beta_i}).
$$
Let $\alpha=\sum \alpha_i$ and $\beta=\sum \beta_i$. We have
\begin{eqnarray*}
e(w_1)\ldots e(w_k)&=& e(u^{\alpha_1})\ldots e(u^{\alpha_{k-1}})e(w_k)\\
                   &\simeq&\partial(u^{\alpha})\partial(w_{k-1}^{-1}\ldots w_1^{-1}w^{\prime}_1\ldots w^{\prime}_l)\\
                   &=&\partial(u^{\alpha})\partial(u^{-\alpha+\beta})\\
                   &\simeq& \partial(u^{\alpha})\partial(u^{-\alpha})\partial(u^{\beta})\\
                   &\simeq& \partial(u^{\beta})\\
                   &\simeq&\partial(u^{\beta_1})\ldots \partial(u^{\beta_l})\\
                   &=&e(w^{\prime}_1)\ldots e(w^{\prime}_l).
\end{eqnarray*}
\end{proof}

\begin{lemma}
Suppose $w_1, \ldots, w_k\in B$. Then\\
1- $e(w_1)\ldots e(w_k)\simeq 1$ implies $w_1\ldots w_k=1$.\\
2- $e(w_1)\ldots e(w_k)\simeq_{\ast} 1$ implies $w_1\ldots w_k\in K$.
\end{lemma}

\begin{proof}
Every cancelation of the first type, corresponds to a cancelation of the form $w_iw_i^{-1}$. Similarly, a cancelation of the second type,
corresponds to a trivial replacement of the form $w^{\alpha+\beta}\leftrightarrow w^{\alpha}w^{\beta}$. This observation proves 1.
To prove 2, we use induction on $k$. The case $k=0$ is trivial, so let $k>0$. Let
$$
e(w_1)\ldots e(w_k)\simeq_{\ast} 1.
$$
By the part 1, it is enough to consider cancelations of type
$$
\partial(w), \ w\in K.
$$
Suppose
$$
e(w_1)\ldots e(w_k)=e(w_1)\ldots e(w_i)\partial(w)e(w_j)\ldots e(w_k),
$$
where $w\in K$ and
$$
\partial(w)=e(w_{i+1})\ldots e(w_{j-1}).
$$
So, we have the next two relations:

i- \ $e(w_1)\ldots e(w_i)e(w_j)\ldots e(w_k)\simeq_{\ast} 1$.

ii- $e(w_{i+1})\ldots e(w_{j-1})\simeq_{\ast} 1$.\\
If $i\neq 0$ or $j-1\neq k$, then the inductive hypothesis can be applied to get
$$
w_1\ldots w_iw_j\ldots w_k,\ \  w_{i+1}\ldots w_{j-1}\in K.
$$
Since $K$ is normal,  we  have $w_1\ldots w_k\in K$. If $i=0$ and $j-1=k$, then we have
$$
\partial(w)=e(w_1)\ldots e(w_k).
$$
Let $w=w^{\prime}_1\ldots w^{\prime}_l$ be the reduced form of $w$. So, we have
$$
e(w^{\prime}_1)\ldots e(w^{\prime}_l)=\partial(w)=e(w_1)\ldots e(w_k),
$$
which implies that $k=l$ and $w_i=w^{\prime}_i$, for any $i$. Hence
$$
w_1\ldots w_k=w^{\prime}_1\ldots w^{\prime}_l=w\in K.
$$
\end{proof}

\begin{corollary}
We have $\partial(w)\simeq_{\ast} 1$ if and only if $w\in K$.
\end{corollary}

\begin{proof}
If $w\in K$, then obviously $\partial(w)\simeq_{\ast}1$. To prove the converse, suppose $w=w_1\ldots w_k$ is the reduced form of $w$.
If $\partial(w)\simeq_{\ast}1$, then $e(w_1)\ldots e(w_k)\simeq_{\ast} 1$, and so by the above lemma
$$
w=w_1\ldots w_k\in K.
$$
\end{proof}

Now, we are ready to prove the main theorem of this section.

\begin{theorem}
$C(X, K)$ is a colored rational complex and its fundamental group is isomorphic to $G$.
\end{theorem}

\begin{proof}
The requirements of the height function can be checked routinely. We must prove the implication
$$
p_1^m\simeq_{\ast} p_2^m \Rightarrow p_1\simeq_{\ast} p_2,
$$
or equivalently the implication
$$
\partial(w_1^m)\simeq_{\ast} \partial(w_2^m) \Rightarrow \partial(w_1)\simeq_{\ast} \partial(w_2).
$$
We have $\partial(w_1^{-m}w_2^m)\simeq_{\ast} 1$, so by the above corollary, $w_1^{-m}w_2^m\in K$. This shows that
$$
(w_1K)^m=(w_2K)^m,
$$
and since, $F_{\mathbb{Q}}(X)/K$ is a rational group, so $w_1K=w_2K$. Therefore $w_1^{-1}w_2\in K$,
and hence $\partial(w_1)\simeq_{\ast}\partial(w_2)$. This proves that $C(X,K)$ is a colored rational complex.
To prove the second assertion, define a map
$$
\varphi: \pi_{\mathbb{Q}}(C, \Phi, v)\to G,
$$
by $\varphi([\partial(w)]_{\ast})=wK$. Now, it can be easily verified that $\varphi$ is a well-defined isomorphism between
the fundamental group of $C(X, K)=(C, \Phi, v)$ and $G=F_{\mathbb{Q}}(X)/K$.
\end{proof}


\section{Covering complex}
In this section, we introduce {\em covering complexes} to study rational subgroups of $F_{\mathbb{Q}}(X)$. Our main aim is to prove that
the variety of rational exponential groups is a Schreier variety. The idea behind our method is the same as in the classical proof of
Nielsen-Schreier's theorem by means of 1-dimensional complexes \cite{Lyndon-Shupp}.

\begin{definition}
Let $(C, v)$ be a rational bouquet. Let $C^{\prime}$ be an arbitrary rational complex. By $\Sigma$ we denote the set of all vertices,
edges and cycles of $C^{\prime}$. Suppose a surjective  map $f:\Sigma\to C(v)$ is given in such a way that\\

1- for any vertex $v^{\prime}$, we have $f(v^{\prime})=v$.

2- if $e^{\prime}$ is an edge, then $f(e^{\prime})$ is also an edge.

3- we have
$$
f(e^{\prime}_1e^{\prime}_2\ldots e^{\prime}_k)=f(e^{\prime}_1)f(e^{\prime}_2)\ldots f(e^{\prime}_k).
$$

4- for any vertex $v^{\prime}\in C^{\prime}$ and $p\in C^{\prime}(v^{\prime})$, the equality
$$
f(p^{(\alpha)})=(f(p))^{(\alpha)}
$$
holds for any $\alpha\in \mathbb{Q}^+$.

5- $f$ is locally one-one, i.e. it is injective on the neighborhood of any vertex.\\
Then we say that $C^{\prime}$ is a {\em covering complex} for $(C, v)$ and $f:C^{\prime}\to C$ is called a {\em covering map}.
\end{definition}

\begin{theorem}
Let $f: C^{\prime}\to (C, v)$ be a covering map. Then for any $v^{\prime}$, there is an induced embedding
$$
f^{\ast}:\pi_{\mathbb{Q}}(C^{\prime}, v^{\prime})\to \pi_{\mathbb{Q}}(C, v).
$$
\end{theorem}

\begin{proof}
Define $f^{\ast}$ by
$$
f^{\ast}([e^{\prime}_1e^{\prime}_2\ldots e^{\prime}_n])=[f(e^{\prime}_1)f(e^{\prime}_2)\ldots f(e^{\prime}_n)].
$$
We show that $f^{\ast}$ is well-defined. We verify that any cancelation in $p^{\prime}=e^{\prime}_1e^{\prime}_2\ldots e^{\prime}_n$ corresponds
to a similar cancelation in $p=f(e^{\prime}_1)f(e^{\prime}_2)\ldots f(e^{\prime}_n)$. If $e^{\prime}e^{\prime -1}$ appears in $p^{\prime}$,
then $f(e^{\prime})f(e^{\prime})^{-1}$ appears in $p$. If $u^{(\alpha)}u^{(\beta)}$ is a part of $p^{\prime}$, then
$$
f(u^{(\alpha)})f(u^{(\beta)})=(f(u))^{(\alpha)}(f(u))^{(\beta)}
$$
is a part of $p$ and so $f^{\ast}$ is well-defined. Clearly, $f^{\ast}$ is a homomorphism. To show that it is injective, suppose
$$
f^{\ast}([e^{\prime}_1e^{\prime}_2\ldots e^{\prime}_n])=1.
$$
In other words, we must show that $f(e^{\prime}_1)f(e^{\prime}_2)\ldots f(e^{\prime}_n)\simeq 1$ implies $e^{\prime}_1e^{\prime}_2\ldots e^{\prime}_n\simeq 1$.

1- Let $f(e^{\prime -1}_i)=f(e^{\prime}_{i+1})$. Since $e^{\prime}_i$ and $e^{\prime}_{i+1}$ have a common terminal point, and $f$ is
locally injective,  $e^{\prime -1}_i=e^{\prime}_{i+1}$.

2- Suppose
$$
u^{(\alpha)}u^{(\beta)}=(f(e^{\prime}_i)\ldots f(e^{\prime}_j))(f(e^{\prime}_{j+1})\ldots f(e^{\prime}_r))
$$
is a part of $f(e^{\prime}_1)f(e^{\prime}_2)\ldots f(e^{\prime}_n)$. We know that $f$ is surjective, so there is  $q\in C^{\prime}(v^{\prime})$,
such that $f(q)=u$. So,
$$
f(q^{(\alpha)})=u^{(\alpha)}=f(e^{\prime}_i\ldots e^{\prime}_j),
$$
and since $f$ is locally injective, we have $q^{(\alpha)}=e^{\prime}_i\ldots e^{\prime}_j$ and a similar argument shows that
$q^{(\beta)}=e^{\prime}_{j+1}\ldots e^{\prime}_r$. So, we have $e^{\prime}_1e^{\prime}_2\ldots e^{\prime}_n\simeq 1$.
\end{proof}

We are now going to prove our main theorem.

\begin{theorem}
The variety of rational exponential groups is a Schreier variety.
\end{theorem}

\begin{proof}
Since any free element of the variety of rational exponential groups has the form $\pi_{\mathbb{Q}}(C, v)$ for some rational bouquet $(C,v)$,
we prove equivalently the following statement:\\

{\em Let $(C, v)$ be a rational bouquet and $H$ be a rational subgroup (i.e. a $\mathbb{Q}$-subgroup) of $\pi_{\mathbb{Q}}(C, v)$.
Then there exists a rational complex $C^{\prime}$, a covering map $f:C^{\prime}\to C$ and a vertex $v^{\prime}\in C^{\prime}$, such that
$$
f^{\ast}(\pi_{\mathbb{Q}}(C^{\prime}, v^{\prime}))=H.
$$
}

This will prove our theorem, since we know that $\pi_{\mathbb{Q}}(C^{\prime}, v^{\prime})$ is free and $f^{\ast}$ is a  monomorphism.
We introduce $C^{\prime}$ as follows.\\

1- The vertices are right cosets $H[p]$, with $[p]\in  \pi_{\mathbb{Q}}(C, v)$.

2- Edges are  pairs $(e, H[p])$, where $e$ is an edge of $C$ and $[p]\in \pi_{\mathbb{Q}}(C, v)$.

3- We define $(e, H[p])^{-1}=(e^{-1}, H[pe])$.

4- We have $in(e, H[p])=H[p]$ and $end(e, H[p])=H[pe]$.

5- Suppose $P=(e_1, H[p_1])(e_2, H[p_2])\ldots (e_n, H[p_n])$ is a cycle with the terminal point $H[p_1]$ and $\alpha\in \mathbb{Q}^+$. Let
$$
(e_1e_2\ldots e_n)^{(\alpha)}=f_1f_2\ldots f_m.
$$
Then, we define
$$
P^{(\alpha)}=(f_1, H[p_1])(f_2, H[p_1f_1])\ldots (f_m, H[p_1f_1\ldots f_{m-1}]).
$$
In the sequel we will show that this action is well-defined.

6- The height function is defined by
$$
h_{H[p]}((e_1, H[p_1])(e_2, H[p_2])\ldots (e_n, H[p_n]))=h_v(e_1e_2\ldots e_n).
$$

First, we show that the action $P^{(\alpha)}$ is well-defined. Since $P$ is a cycle, so for any $i$, we have
$$
end(e_i, H[p_i])=in(e_{i+1}, H[p_{i+1}]),
$$
hence $H[p_ie_i]=H[p_{i+1}]$. This shows that
$$
[p_1][e_1e_2\ldots e_n][p_1]^{-1}\in H,
$$
and since $H$ is a rational subgroup, we obtain
$$
[p_1][f_1f_2\ldots f_m][p_1]^{-1}=[p_1][e_1e_2\ldots e_n]^{\alpha}[p_1]^{-1}\in H.
$$
So, the action is well-defined. To make working with cycles in $C^{\prime}$ easier, we introduce a notation: if
$$
P=(e_1, H[p_1])(e_2, H[p_2])\ldots (e_n, H[p_n])
$$
is a cycle, with the end point $H[p_1]=H[p]$, then we can re-write it in a new form
$$
P=(e_1, H[p])(e_2, H[pe_1])\ldots (e_n, H[pe_1\ldots e_{n-1}]).
$$
Further, we have $[p][e_1e_2\ldots e_n][p]^{-1}\in H$. So, we may use the notation
$$
P=\langle e_1e_2\ldots e_n, H[p]\rangle.
$$
With this new notation, we have
$$
P^{(\alpha)}=\langle (e_1e_2\ldots e_n)^{(\alpha)}, H[p]\rangle.
$$
This notation has some other advantages: if
$$
P=\langle e_1e_2\ldots e_n, H[p]\rangle, \ \ and\ \ Q=\langle f_1f_2\ldots f_m, H[p]\rangle,
$$
then it is easy to see that
$$
PQ=\langle e_1\ldots e_nf_1\ldots f_m, H[p]\rangle.
$$
Now, for example, we have
\begin{eqnarray*}
P^{(\alpha)}P^{(\beta)}&=&\langle (e_1\ldots e_n)^{(\alpha)}(e_1\ldots e_n)^{(\beta)}, H[p]\rangle\\
                       &\simeq&\langle (e_1\ldots e_n)^{(\alpha+\beta)}, H[p]\rangle.
\end{eqnarray*}
Let $P_1=\langle p_1, H[p]\rangle$ and $P_2=\langle p_2, H[p]\rangle$ and suppose $P_1^m\simeq P_2^m$. Then we have also $p_1^m\simeq p_2^m$,
and hence $p_1\simeq p_2$. This shows that $P_1\simeq P_2$. We can also check the requirements of height function. For example, if
$$
\langle e_1e_2\ldots e_n, H[p]\rangle\simeq \langle f_1f_2\ldots f_m, H[p]\rangle,
$$
then $e_1e_2\ldots e_n\simeq f_1f_2\ldots f_m$ and so $h(e_1e_2\ldots e_n)=h(f_1f_2\ldots f_m)$. This shows that $h_{H[p]}$ is
constant on homotopy classes of $C^{\prime}$. Other conditions for height function can be verified similarly. Therefore, $C^{\prime}$ is a
rational complex. Now, define a map $f:C^{\prime}\to C$ by $f(H[p])=v$ and $ f(\langle e_1e_2\ldots e_n, H[p]\rangle)=e_1e_2\ldots e_n$.
Clearly $f$ is surjective and it is locally injective. Properties of covering map can be checked for $f$. Suppose $v^{\prime}=H$.
We prove that $f^{\ast}(\pi_{\mathbb{Q}}(C^{\prime}, v^{\prime}))=H$. Note that
$$
C^{\prime}(v^{\prime})=\{ \langle p, H\rangle : p\in C(v), [p]\in H\}.
$$
Hence
$$
\pi_{\mathbb{Q}}(C^{\prime}, v^{\prime})=\{ [\langle p, H\rangle]:\ [p]\in H\},
$$
and therefore
$$
f^{\ast}(\pi_{\mathbb{Q}}(C^{\prime}, v^{\prime}))=\{ [p]:\ [p]\in H\}=H.
$$
This completes the proof.
\end{proof}


\begin{thebibliography}{99}

\bibitem{Baum} Baumslag, G.
{\it Some aspects of groups with unique roots}, Acta Math., Vol. 104,  pp. 217-303, 1960.

\bibitem{Hall} Hall, P.
{\it Nilpotent groups}, Canadian Math. Congress, Edmonton, 1957.

\bibitem{Olga1} Kharlampovich, O.,   Lioutikova, E.,  Myasnikov, A. G.
{\it Equations in the $\mathbb{Q}$-completion of a torsion-free hyperbolic group}, Transactions of the A.M.S., Vol. 351, No. 4, pp. 2961-2978, 1999.

\bibitem{Olga2} Kharlampovich, O.,    Myasnikov, A. G.
{\it Equations in a free $\mathbb{Q}$-group}, Transactions of the A.M.S., Vol. 350, No. 3, pp. 947-974, 1998.

\bibitem{Olga3} Kharlampovich, O.,    Myasnikov, A. G.
{\it Irreducible affine varieties over a free group. I: Irreduciblity of quadratic equations and nullstellensatz},
J. Algebra, Vol. 200, No. 2, pp. 472-516, 1998.

\bibitem{Olga4} Kharlampovich, O.,    Myasnikov, A. G.
{\it Irreducivle affine varieties over free groups. II: Systems in triangular quasi-quadratic form and description of residually free groups},
J. Algebra, Vol. 200, No. 2, pp. 517-570, 1998.

\bibitem{Lewin} Lewin J.
{\it On Schreier varieties of linear algebras}, Transactions of the A.M.S., Vol. 132, No. 2, pp. 553-562, 1968.

\bibitem{Lyndon} Lyndon R.
{\it Groups with parametric exponents}, Transactions of the A.M.S., Vol. 96, No. 6, pp. 518-533, 1960.


\bibitem{Lio-Myas} Lioutikov, S., Myasnikov, A. G.
{\it Centroid of groups}, J. Group Theory., Vol. 3, No. 2, pp. 177-197, 2000.


\bibitem{Lyndon-Shupp} Lyndon R., Schupp P.
Combinatorial group theory, Springer-Verlag, 1977.

\bibitem{Mal} Mal'cev, A. I.
{\it Torsion free nilpotent groups}, Izv. AN SSSR, Ser. Math. Vol. 13, No. 3, pp. 201-212, 1949.

\bibitem{Polin} Polin, S. V. 
{\it Free decompositions in varieties of $\Lambda$-groups}, Mat. Sb., 87(129), pp.  377–395, 1972.

\bibitem{MR1}
Myasnikov A. G., Remeslennikov V. N.  {\it Exponential groups I: fundations of the theory and tensor completions}, Siberian Math. J.
Vol 35, No. 5, pp. 986-996, 1994.

\bibitem{MR2}
Myasnikov A. G., Remeslennikov V. N. {\it Exponential groups II: extensions of centralizers and tensor completion of CSA-groups},
International J. Algebra and Computation, Vol. 6, No. 6, pp. 687-711, 1996.


\bibitem{Rem}
Remeslennikov V. N.  {\it E-free groups}, Siberian Math. J. Vol. 30, No. 6, pp. 998-1001, 1989.


Polin (, Mat. Sb. 87(129), (1972)).




\end{thebibliography}
\end{document}